\documentclass[12pt,reqno,oneside]{amsart}
\usepackage{etoolbox}

\makeatletter
\patchcmd{\enddoc@text}
  {\@setaddresses}
  {}
  {}{}
\apptocmd{\@maketitle}
  {\vspace{0.1em}\@setaddresses}
  {}{}
\makeatother

\usepackage[T2A]{fontenc}
\usepackage[utf8]{inputenc}
\usepackage[english]{babel}
\usepackage[overload]{textcase}

\usepackage[backend=biber,style=ieee,autolang=other]{biblatex}
\AtEveryBibitem{%
  \ifentrytype{article}{%
    \clearfield{url}%
    \clearfield{urlyear}%
    \clearfield{urlmonth}%
    \clearfield{urlday}%
    \clearfield{issn}%
  }{}%
  \ifentrytype{inproceedings}{%
    \clearfield{url}%
    \clearfield{urlyear}%
    \clearfield{urlmonth}%
    \clearfield{urlday}%
  }{}%
}

\usepackage{amsmath, amssymb, amsthm, amsfonts}
\usepackage{geometry}
\usepackage[colorlinks=true, linkcolor=blue, citecolor=red, urlcolor=blue]{hyperref}
\usepackage{graphicx}
\graphicspath{{./images/}}
\usepackage{xcolor}
\usepackage{enumitem}
\usepackage{comment}

\theoremstyle{plain}
\newtheorem{theorem}{Theorem}[section]

\newtheorem{proposition}[theorem]{Proposition}

\theoremstyle{definition}
\newtheorem{definition}[theorem]{Definition}

\theoremstyle{remark}
\newtheorem{remark}[theorem]{Remark}

\numberwithin{equation}{section}

\begin{document}

\title[The Partition of Paley Graphs into Petersen Graphs]{The Partition of Paley Graphs into Petersen Graphs and a New Strongly Regular Graph with Parameters (50, 21, 8, 9)}

\author{Viktor A. Byzov}
\address{Vyatka State University, Kirov, Russia}
\email{vbyzov@yandex.ru}


\date{\today}

\begin{abstract}
This paper utilizes an extremely simple idea: in the multiplicative group of a finite field $F_q$, cosets of a certain subgroup are considered, and an attempt is made to combine these cosets into pairs such that the resulting induced subgraph in the corresponding Paley graph is strongly regular. It is shown that there exists an infinite sequence of Paley graphs that can be partitioned into Petersen graphs in this manner. Furthermore, a new strongly regular graph with parameters $(50, 21, 8, 9)$ is constructed using this method.

\end{abstract}

\maketitle

\section{Introduction}
\label{sec:intro}

In graph theory, the area related to partitioning the vertex set into induced subgraphs of a given type is actively developing. We note the classical paper by W.\,H.~Haemers and D.\,G.~Higman~\cite{haemers_strongly_1989}, which investigates the partition of the vertex set $V = V_1 \cup V_2$ of a strongly regular graph such that both induced subgraphs $G[V_1]$ and $G[V_2]$ are themselves strongly regular graphs, cliques, or independent sets (cocliques). The paper presents methods for constructing graphs that admit a strongly regular decomposition, and also provides proofs of the non-existence of decompositions for certain parameter sets. In addition, the article contains a summary table of all feasible parameter sets for graphs with up to 300 vertices, in which the question of their existence is definitively resolved for most cases.

In the work of A.\,D.~Sankey~\cite{sankey_strongly_2021}, the ideas of Haemers and Higman are developed further: the author investigates strongly regular designs that admit fusion into a strongly regular graph. The author applies elementary methods to derive decomposition conditions equivalent to the theorem of Haemers and Higman from~\cite{haemers_strongly_1989}. The paper presents an updated table of feasible parameters for this class of configurations, analyzes graph constructions, and proves conditions for their non-existence. The topic of the decomposition of strongly regular graphs is also addressed in the works of A.\,L.~Gavrilyuk and V.\,V.~Kabanov~\cite{gavrilyuk_strongly_2023} and~\cite{gavrilyuk_strongly_2025}. These papers investigate strongly regular graphs that can be decomposed into a divisible design graph and a Hoffman coclique or a Delsarte clique.

In this paper, we investigate the possibility of representing the vertex set of a Paley graph in the form $V = V_0 \cup V_1 \cup \ldots \cup V_k$, where $V_0$ is a single-element set containing the zero vertex, and $V_i$ for $i \geq 1$ are vertex sets such that the induced subgraphs $G[V_i]$ are isomorphic to the Petersen graph. Here, the vertex sets $V_i$ are obtained by combining cosets of a subgroup of order five in the multiplicative group of a finite field into pairs. Additionally, in this work, the idea of combining pairs of cosets is used to construct a new strongly regular graph with parameters $(50, 21, 8, 9)$.

\section{Preliminaries}
\label{sec:preliminaries}
We briefly discuss the basic concepts and statements necessary for the further exposition.
\subsection{Weil bound for multiplicative characters}

Let $q$ be a positive integer power of a prime $p$. We denote by $\mathbb{F}_q$ the finite field with $q$ elements, and by $\mathbb{F}_q^{\times}$ its multiplicative group.

We will use the symbol $\chi$ to denote the quadratic character of the field $\mathbb{F}_q$. For any $a\in \mathbb{F}_{q}^{\times}$,
\begin{equation}
\chi(a) = \begin{cases}
    1, \; \text{ if } \; a \in \left(\mathbb{F}_{q}^{\times}\right)^2, \\
    -1, \; \text{ otherwise.}
\end{cases}
\end{equation}

In this paper, we will need the following statement.

\begin{theorem}[{Weil, 1948 \cite[p.~225]{lidl_finite_1996}}]
\label{th:weil}
Let $\psi$ be a multiplicative character of the field $\mathbb{F}_q$ of order $m > 1$, and let $f\in \mathbb{F}_q[x]$ be a monic polynomial of positive degree that is not the $m$-th power of another polynomial. If $d$ is the number of distinct roots of the polynomial $f$ in its splitting field over $\mathbb{F}_q$, then for every $a\in \mathbb{F}_q$ the following inequality holds:
\begin{equation}
\left|\sum_{c\in \mathbb{F}_q} \psi(a f(c))\right| \leq (d - 1) q^{1/2}.
\end{equation}
\end{theorem}

\subsection{Non-representability of prime numbers by quadratic forms}

In this paper, we will also need the following statement.

\begin{theorem}
\label{th:square_forms}
Let $p$ be a prime number such that $p \equiv 1 \pmod{20}$. We introduce the following notation: let $\alpha$ be a primitive element of the field $\mathbb{F}_p$, and $h = \alpha^{\frac{p-1}{5}}$. The following conditions are equivalent:
\begin{enumerate}
\item $\chi(1+h) = -1$;
\item $5 \notin \langle \alpha \rangle^4$;
\item there exist no $(x, y) \in \mathbb{Z}^2$ such that $p = x^2 + 20y^2$;
\item there exist no $(x, y) \in \mathbb{Z}^2$ such that $p = x^2 + 100y^2$.
\end{enumerate} 
\end{theorem}

The equivalence of conditions (1) and (2) was proved by A.~Hanaki, K.~Kobayashi, and A.~Munemasa in~\cite[lemma~4.5]{hanaki_3-designs_2026}. The equivalence of conditions (3) and (4) was proved by D.~Brink in~\cite[theorem~1]{brink_five_2009}. The equivalence of conditions (2) and (4) was proved by H.~Hasse in~\cite[p.~69]{hasse_bericht_1965}. The formulation of Theorem~\ref{th:square_forms} itself and the references to the sources are borrowed from~\cite{hanaki_3-designs_2026}.

\subsection{Strongly regular graphs}

\begin{definition}
A \textit{strongly regular graph} with parameters $(v, k, \lambda, \mu)$ is a graph on $v$ vertices for which the following conditions hold:
\begin{enumerate}
\item the degree of each vertex is $k$;
\item for any two adjacent vertices $x_1$ and $x_2$, there are exactly $\lambda$ paths of length two from $x_1$ to $x_2$;
\item for any two non-adjacent vertices $x_1$ and $x_2$, there are exactly $\mu$ paths of length two from $x_1$ to $x_2$.
\end{enumerate}
\end{definition}

Instead of the long phrase ``a strongly regular graph with parameters $(v, k, \lambda, \mu)$'', the following shorter notation is often used: $\text{srg}(v, k, \lambda, \mu)$. As an example of a strongly regular graph, one can mention the Petersen graph; it is an $\text{srg}(10, 3, 0, 1)$.

\subsection{Paley graphs}

Let $q$ be a prime power such that $q \equiv 1 \pmod 4$. The \textit{Paley graph $P(q)$} of order $q$ is a graph whose vertices are the elements of the field $\mathbb{F}_q$, and two vertices $x_1$ and $x_2$ are connected by an edge if and only if $\chi(x_1 - x_2) = 1$. Paley graphs are undirected because $\chi(x_1-x_2) = \chi(x_2-x_1)$ when $q \equiv 1 \pmod 4$. The Paley graph $P(q)$ is strongly regular with parameters \mbox{$(q, (q - 1)/2, (q - 5)/4, (q - 1)/4)$}.

\subsection{Two-graphs}

Two-graphs are not graphs in the usual sense of the term, but rather incidence structures.

\begin{definition}
A \textit{two-graph} $\Phi$ is an ordered pair $(V, T)$, where $V$ is a set of vertices and $T$ is a collection of three-element subsets (called triples) of the set $V$, such that every four-element subset of $V$ contains (as subsets) an even number of triples from~$T$.

A two-graph is called \textit{regular} if every pair of vertices is contained in the same number of triples.
\end{definition}

Any simple undirected graph $G = (V, E)$ uniquely generates a two-graph on the same vertex set: a triple of vertices is included in $T$ if it induces a subgraph in $G$ with an odd number of edges (one or three edges). 

A central tool in the theory of two-graphs is the operation of Seidel switching.

\begin{definition}
Let $G = (V, E)$ be a graph, and $W \subset V$ be an arbitrary subset of its vertices. \textit{Seidel switching} with respect to $W$ creates a new graph $G'$ in which the adjacency between any vertex in $W$ and any vertex in the complement $V \setminus W$ is reversed (edges become non-edges, and non-edges become edges). The adjacency of pairs of vertices entirely contained within $W$ or within $V \setminus W$ remains unchanged.

Two graphs are called \textit{Seidel equivalent} (or belonging to the same switching class) if one can be obtained from the other by applying this operation.
\end{definition}

An important property of the Seidel switching operation is the following fact (see, e.g.,~\cite{cameron_designs_2007}).

\begin{theorem}
Two graphs generate the same two-graph if and only if they are Seidel equivalent. Thus, a two-graph can be considered as an invariant of a switching class of graphs.
\end{theorem}

\section{Results}
\label{sec:main}

\subsection{Partition of Paley graphs into Petersen graphs}

Let $\mathbb{F}_q$ be a finite field with $q$ elements, where $q$ is a power of a prime $p$. Additionally, we require that the congruences $q \equiv 1 \pmod 4$ and $q \equiv 1 \pmod{10}$ hold. That is, $q \equiv 1 \pmod{20}$.

Let $\alpha$ denote a primitive element of the field $\mathbb{F}_q$, and let $h = \alpha^{\frac{q-1}{5}}$. Note that since $q \equiv 1 \pmod{20}$, we have $\chi(h) = 1$. Consider the subgroup $H = \{1, h, h^2, h^3, h^4\}$ of the group $\mathbb{F}_q^{\times}$ and the cosets of this subgroup. Since $\chi(h) = 1$, each coset either consists entirely of quadratic residues (we will call such classes \textit{residue orbits}), or consists entirely of quadratic non-residues (we will call such classes \textit{non-residue orbits}).

The author observed that in the Paley graph $P(41)$, four induced Petersen subgraphs can be identified, as shown in Fig.~\ref{fig:paley41}. Here, the outer ``pentagons'' of the Petersen graphs correspond to the residue orbits, while the inner ``stars'' correspond to the non-residue orbits.

\begin{figure}[htbp]
    \centering
    \includegraphics[width=0.6\textwidth]{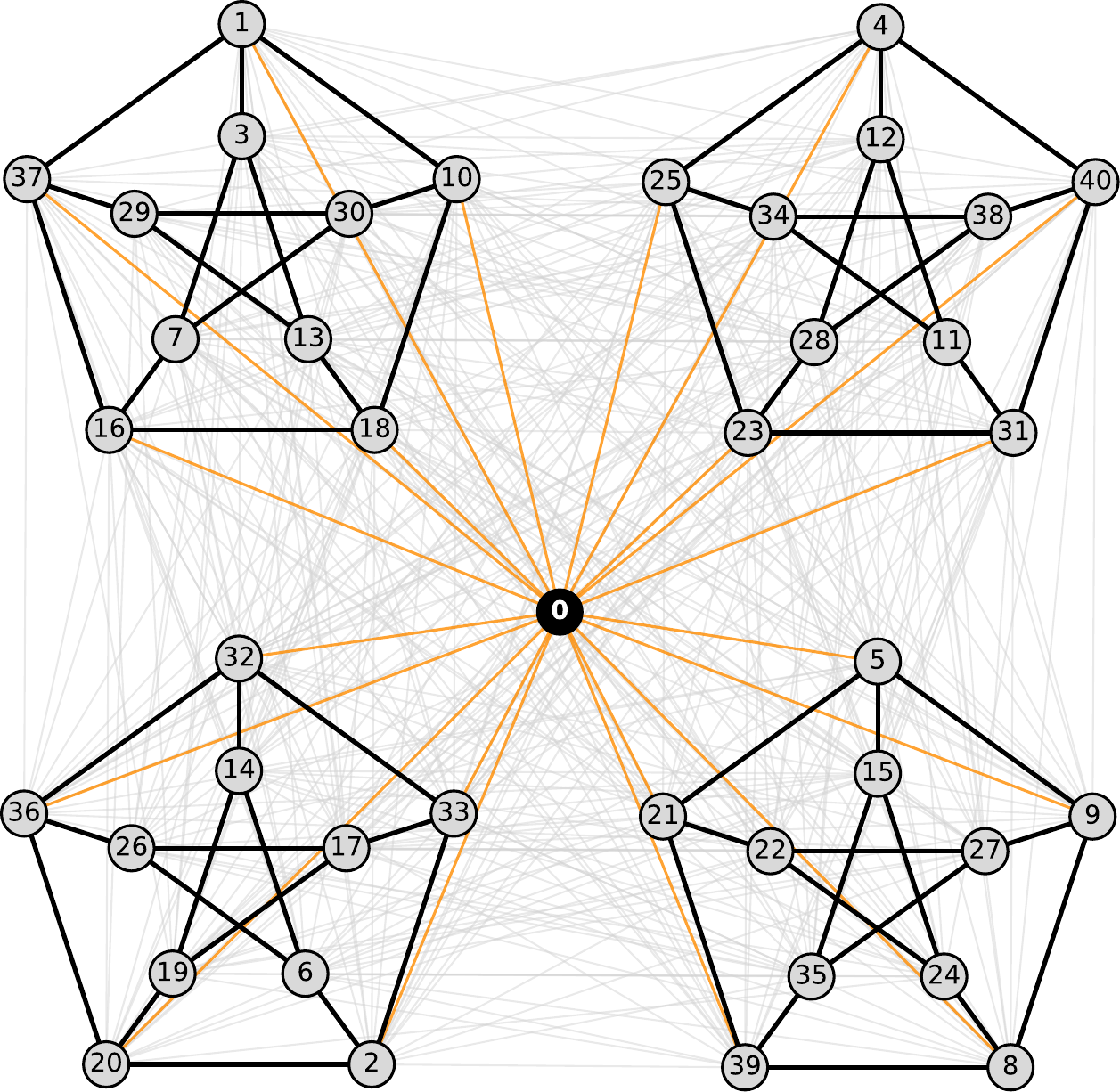}
    \caption{Induced Petersen subgraphs in the graph $P(41)$}
    \label{fig:paley41}
\end{figure}

We call the Paley graph $P(q)$ (for $q \equiv 1 \pmod{20}$) \textit{$G(5, 2)$-partitionable} if the residue orbits described above can be paired with the non-residue orbits such that the induced subgraphs corresponding to these pairs are isomorphic to the Petersen graph. The only vertex not included in the Petersen subgraphs is the zero vertex.

\begin{theorem}
\label{th:paley_petersen}
    The Paley graph $P(q)$ (for $q \equiv 1 \pmod{20}$) is $G(5, 2)$-partitionable if and only if $\chi(1+h) = -1$.
\end{theorem}
\begin{proof}
Below in the proof, we will use the fact that $q > 315^2$. For smaller field sizes, the validity of the theorem was verified by computer simulation.

Assume that $\chi(1+h) = 1$. Consider the residue orbit $\{1$, $h$, $h^2$, $h^3$, $h^4\}$. Then $\chi(1-h^2) = \chi(1-h)\cdot\chi(1+h) = \chi(1-h)$. That is, for the three vertices $1, h, h^2$ of the Paley graph, the following holds: either these three vertices form a $K_3$, or they form a $\overline{K_3}$. In the first case, the induced subgraph cannot be a Petersen graph, since the Petersen graph contains no triangles. In the second case, since
\begin{equation}
\label{eq:chi_prop}
\chi(h^a - h^{a+b})=\chi(1-h^b),
\end{equation}
\noindent the vertices $1$, $h$, $h^2$, $h^3$, $h^4$ form an independent set. But the independence number of the Petersen graph is $4$, so the induced subgraph containing this orbit is not isomorphic to the Petersen graph. Thus, we have shown that if the graph is $G(5, 2)$-partitionable, then $\chi(1+h) = -1$.

Conversely, suppose that $\chi(1+h) = -1$. Since $\chi(1-h^2) = \chi(1-h)\cdot\chi(1+h) = -\chi(1-h)$, there is an edge from the vertex $1$ to exactly one of the two vertices: $h$ or $h^2$. It follows from relation~\eqref{eq:chi_prop} that in the first case we obtain the cycle $1 - h - h^2 - h^3 - h^4 - 1$, and in the second case, the cycle $1 - h^2 - h^4 - h - h^3 - 1$.

Let $c$ be some quadratic non-residue; consider the non-residue orbit $\{c$, $ch$, $ch^2$, $ch^3$, $ch^4\}$. Note that $\chi(ch^a - ch^b) = -\chi(h^a - h^b)$, therefore the edges within this non-residue orbit are the complements of the edges from the orbit $\{1$, $h$, $h^2$, $h^3$, $h^4\}$. Thus, we automatically obtain the inner ``star'' of the Petersen graph, which is the complement to the outer ``pentagon''. It remains only to find a $c$ such that this ``star'' and ``pentagon'' are connected by five suitable edges.

We seek a quadratic non-residue $c$ such that vertex $1$ is adjacent to vertex $c$ and not adjacent to vertices $ch$, $ch^2$, $ch^3$, $ch^4$. That is,
\begin{align}
\chi(1-c) &= 1,\\
\chi(1-ch^j) &= -1 \; \text{ for } \; j\in\{1, 2, 3, 4\}.
\end{align}
\noindent Note that in this case, there will be exactly one edge from vertex $h^{j}$ to vertex~$ch^{j}$ (for $j\in\{1, 2, 3, 4\}$).

Let us introduce the indicator function $I(c)$, which equals one if $c$ is suitable, and zero otherwise:
\begin{equation}
I(c) = \frac{1 - \chi(c)}{2} \cdot \frac{1 + \chi(1-c)}{2} \cdot \prod_{j=1}^4 \frac{1 - \chi(1-ch^j)}{2}.
\end{equation}

Let us sum the indicators over all elements of the field:
\begin{equation}
N = \sum_{c \in \mathbb{F}_q} I(c).
\end{equation}

Note that expanding the brackets in the indicator $I(c)$ yields the constant term $\frac{1}{64}$ and 63 terms of the form $\pm\frac{1}{64}\chi(P(c))$, where each $P(x)$ is a polynomial of degree at most six. Since the roots of these polynomials are drawn from the set $\{0, 1, h^{-1}, h^{-2}, h^{-3}, h^{-4}\}$ containing pairwise distinct elements, all roots of $P(x)$ have a multiplicity of one. Therefore, no such polynomial is the square of another polynomial. We can represent each $P(x)$ as $a f(x)$, where $a \in \mathbb{F}_q^{\times}$ is the leading coefficient and $f(x)$ is a monic polynomial. By Theorem~\ref{th:weil}, summing these 63 terms over all elements of the field $\mathbb{F}_q$ gives a value not exceeding $\frac{63\cdot 5}{64} q^{1/2}$ in absolute value. Thus, the following inequality holds:
\begin{equation}
N \geq \frac{q}{64} - \frac{315}{64}q^{1/2}.
\end{equation}

It is easy to see that for $q > 315^2$, the value $N$ is greater than zero, meaning that the desired quadratic non-residue $c$ certainly exists (the inequality for $q$ was discussed at the beginning of the proof). With this value of $c$, the residue orbit $\{1, h, h^2, h^3, h^4\}$ together with the non-residue orbit $\{c, ch, ch^2, ch^3, ch^4\}$ form an induced subgraph isomorphic to the Petersen graph.

Note that the found value of $c$ allows us to find a suitable pair for any other residue orbit $\{a, ah, ah^2, ah^3, ah^4\}$, where $\chi(a) = 1$. Such a non-residue orbit would be \{$ca$, $cah$, $cah^2$, $cah^3$, $cah^4$\}. This is true because multiplying the values of all vertices by a quadratic residue does not change the adjacency of the vertices. Thus, it is proved that for $\chi(1 + h) = -1$, the graph $P(q)$ is $G(5, 2)$-partitionable.
\end{proof}

\begin{remark}
Note that the suitable quadratic non-residue $c$ in the proof of the theorem is not unique. For example, for the graph $P(241)$, there are 20 different suitable values of $c$. Moreover, one can obtain 4 different partitions of the graph $P(241)$ into Petersen graphs (five different values of $c$ yield the same partition of orbits into pairs).
\end{remark}

If, under the conditions of Theorem~\ref{th:paley_petersen}, we consider only fields whose number of elements is a prime number, we obtain the following sequence of suitable field sizes: 41, 61, 241, 281, 421, 601, 641, 661, 701, 821, \ldots. Note that this sequence has already appeared in the literature: in~\cite{hanaki_3-designs_2026}, finite fields of these sizes were used to construct $3$-designs; in the On-Line Encyclopedia of Integer Sequences, this sequence has the number A325072 (see~\cite{oeis_a325072}).

\begin{proposition}
\label{th:infinity}
The sequence of sizes of finite fields $q$ satisfying the condition of Theorem~\ref{th:paley_petersen} ($q \equiv 1 \pmod{20}$ and $\chi(1+h) = -1$) is infinite.
\end{proposition}
\begin{proof}
We will prove the infinity of the set of field sizes that are prime numbers and satisfy the condition $\chi(1+h) = -1$. It follows from Theorem~\ref{th:square_forms} that instead of this, one can prove the infinity of the set of prime numbers $p \equiv 1 \pmod{20}$ that are not representable by the quadratic form $x^2+20y^2$.

It is well known (see, e.g.,~\cite{cox_primes_2013}) that any prime number $p \equiv 1 \pmod{20}$ can be represented in the form $p = a^2+5b^2$. Note that in this representation, either $a$ is even or $b$ is even. In the first case, we obtain a representation of the form $4x^2+5b^2$; in the second, a representation of the form $a^2+20y^2$. Consequently, the non-representability of a prime number $p \equiv 1 \pmod{20}$ by the form $x^2 + 20y^2$ is equivalent to its representability by the form $4x^2 + 5y^2$.

In turn, according to Meyer's theorem (see, e.g.,~\cite{reiner_generalization_1949}), if a properly primitive binary quadratic form represents at least one number in an arithmetic progression whose first term and common difference are coprime, then it represents infinitely many prime numbers in this progression. Since the positive-definite form $4x^2 + 5y^2$ is properly primitive and represents the number 61, which belongs to the progression $p \equiv 1 \pmod{20}$, it generates infinitely many prime numbers of the specified form.
\end{proof}

\begin{remark}
The proof of Proposition~\ref{th:infinity} discusses the infinity of the set of suitable field sizes that are prime numbers. However, it should be noted that there are also suitable extensions of prime fields, for example, $\mathbb{F}_{41^3}$.
\end{remark}

\subsection{Construction of a new graph $\text{srg}(50, 21, 8, 9)$}

The Petersen graph is not the only strongly regular graph that can be obtained by pairing cosets of $\mathbb{F}_q^{\times}$ with respect to a subgroup of a given order. Using computer simulation, it was found that in the graphs $P(13781)$, $P(20021)$, $P(23869)$, etc., by pairing cosets of a subgroup of order 13, one can obtain induced subgraphs isomorphic to $\text{srg}(26, 10, 3, 4)$. It is known that, up to isomorphism, there are 10 strongly regular graphs with parameters $(26, 10, 3, 4)$ (see~\cite{brouwer_srg}). Among these 10~graphs, there is one having the full automorphism group $C_{13}\rtimes C_3$ --- this is precisely the graph obtained as a result of such a pairing of orbits.

To facilitate the search for new strongly regular graphs, a custom computational algorithm was implemented in C++ and accelerated via CUDA technology for parallel processing. The program iterates over Paley graphs to identify instances where the pairing of two cosets of a subgroup of order 25 yields an induced subgraph isomorphic to $\text{srg}(50, 21, 8, 9)$. The computations were executed on an NVIDIA T4 GPU within the Google Colaboratory environment, requiring approximately one to two minutes of execution time. The first graph satisfying these criteria was found to be $P(1117901)$.

As a primitive element of the field $\mathbb{F}_{1117901}$, we took $\alpha = 2$, and $h = \alpha^{\frac{1117901 - 1}{25}}$; the subgroup $H = \{1, h, h^2, \ldots, h^{24}\}$ was considered. Multiplying all elements of the residue orbit containing one by the quadratic non-residue $754583$ yields the paired non-residue orbit. The adjacency matrix of the obtained graph is given below.
\begin{equation}
\resizebox{0.9\textwidth}{!}{$
\left(
\begin{array}{ccccccccccccccccccccccccc|ccccccccccccccccccccccccc}
0 & 0 & 1 & 1 & 0 & 1 & 1 & 1 & 0 & 0 & 0 & 0 & 1 & 1 & 0 & 0 & 0 & 0 & 1 & 1 & 1 & 0 & 1 & 1 & 0 & 0 & 1 & 1 & 0 & 0 & 1 & 0 & 0 & 1 & 1 & 0 & 1 & 0 & 1 & 0 & 0 & 0 & 0 & 0 & 0 & 0 & 0 & 1 & 0 & 1 \\
0 & 0 & 0 & 1 & 1 & 0 & 1 & 1 & 1 & 0 & 0 & 0 & 0 & 1 & 1 & 0 & 0 & 0 & 0 & 1 & 1 & 1 & 0 & 1 & 1 & 1 & 0 & 1 & 1 & 0 & 0 & 1 & 0 & 0 & 1 & 1 & 0 & 1 & 0 & 1 & 0 & 0 & 0 & 0 & 0 & 0 & 0 & 0 & 1 & 0 \\
1 & 0 & 0 & 0 & 1 & 1 & 0 & 1 & 1 & 1 & 0 & 0 & 0 & 0 & 1 & 1 & 0 & 0 & 0 & 0 & 1 & 1 & 1 & 0 & 1 & 0 & 1 & 0 & 1 & 1 & 0 & 0 & 1 & 0 & 0 & 1 & 1 & 0 & 1 & 0 & 1 & 0 & 0 & 0 & 0 & 0 & 0 & 0 & 0 & 1 \\
1 & 1 & 0 & 0 & 0 & 1 & 1 & 0 & 1 & 1 & 1 & 0 & 0 & 0 & 0 & 1 & 1 & 0 & 0 & 0 & 0 & 1 & 1 & 1 & 0 & 1 & 0 & 1 & 0 & 1 & 1 & 0 & 0 & 1 & 0 & 0 & 1 & 1 & 0 & 1 & 0 & 1 & 0 & 0 & 0 & 0 & 0 & 0 & 0 & 0 \\
0 & 1 & 1 & 0 & 0 & 0 & 1 & 1 & 0 & 1 & 1 & 1 & 0 & 0 & 0 & 0 & 1 & 1 & 0 & 0 & 0 & 0 & 1 & 1 & 1 & 0 & 1 & 0 & 1 & 0 & 1 & 1 & 0 & 0 & 1 & 0 & 0 & 1 & 1 & 0 & 1 & 0 & 1 & 0 & 0 & 0 & 0 & 0 & 0 & 0 \\
1 & 0 & 1 & 1 & 0 & 0 & 0 & 1 & 1 & 0 & 1 & 1 & 1 & 0 & 0 & 0 & 0 & 1 & 1 & 0 & 0 & 0 & 0 & 1 & 1 & 0 & 0 & 1 & 0 & 1 & 0 & 1 & 1 & 0 & 0 & 1 & 0 & 0 & 1 & 1 & 0 & 1 & 0 & 1 & 0 & 0 & 0 & 0 & 0 & 0 \\
1 & 1 & 0 & 1 & 1 & 0 & 0 & 0 & 1 & 1 & 0 & 1 & 1 & 1 & 0 & 0 & 0 & 0 & 1 & 1 & 0 & 0 & 0 & 0 & 1 & 0 & 0 & 0 & 1 & 0 & 1 & 0 & 1 & 1 & 0 & 0 & 1 & 0 & 0 & 1 & 1 & 0 & 1 & 0 & 1 & 0 & 0 & 0 & 0 & 0 \\
1 & 1 & 1 & 0 & 1 & 1 & 0 & 0 & 0 & 1 & 1 & 0 & 1 & 1 & 1 & 0 & 0 & 0 & 0 & 1 & 1 & 0 & 0 & 0 & 0 & 0 & 0 & 0 & 0 & 1 & 0 & 1 & 0 & 1 & 1 & 0 & 0 & 1 & 0 & 0 & 1 & 1 & 0 & 1 & 0 & 1 & 0 & 0 & 0 & 0 \\
0 & 1 & 1 & 1 & 0 & 1 & 1 & 0 & 0 & 0 & 1 & 1 & 0 & 1 & 1 & 1 & 0 & 0 & 0 & 0 & 1 & 1 & 0 & 0 & 0 & 0 & 0 & 0 & 0 & 0 & 1 & 0 & 1 & 0 & 1 & 1 & 0 & 0 & 1 & 0 & 0 & 1 & 1 & 0 & 1 & 0 & 1 & 0 & 0 & 0 \\
0 & 0 & 1 & 1 & 1 & 0 & 1 & 1 & 0 & 0 & 0 & 1 & 1 & 0 & 1 & 1 & 1 & 0 & 0 & 0 & 0 & 1 & 1 & 0 & 0 & 0 & 0 & 0 & 0 & 0 & 0 & 1 & 0 & 1 & 0 & 1 & 1 & 0 & 0 & 1 & 0 & 0 & 1 & 1 & 0 & 1 & 0 & 1 & 0 & 0 \\
0 & 0 & 0 & 1 & 1 & 1 & 0 & 1 & 1 & 0 & 0 & 0 & 1 & 1 & 0 & 1 & 1 & 1 & 0 & 0 & 0 & 0 & 1 & 1 & 0 & 0 & 0 & 0 & 0 & 0 & 0 & 0 & 1 & 0 & 1 & 0 & 1 & 1 & 0 & 0 & 1 & 0 & 0 & 1 & 1 & 0 & 1 & 0 & 1 & 0 \\
0 & 0 & 0 & 0 & 1 & 1 & 1 & 0 & 1 & 1 & 0 & 0 & 0 & 1 & 1 & 0 & 1 & 1 & 1 & 0 & 0 & 0 & 0 & 1 & 1 & 0 & 0 & 0 & 0 & 0 & 0 & 0 & 0 & 1 & 0 & 1 & 0 & 1 & 1 & 0 & 0 & 1 & 0 & 0 & 1 & 1 & 0 & 1 & 0 & 1 \\
1 & 0 & 0 & 0 & 0 & 1 & 1 & 1 & 0 & 1 & 1 & 0 & 0 & 0 & 1 & 1 & 0 & 1 & 1 & 1 & 0 & 0 & 0 & 0 & 1 & 1 & 0 & 0 & 0 & 0 & 0 & 0 & 0 & 0 & 1 & 0 & 1 & 0 & 1 & 1 & 0 & 0 & 1 & 0 & 0 & 1 & 1 & 0 & 1 & 0 \\
1 & 1 & 0 & 0 & 0 & 0 & 1 & 1 & 1 & 0 & 1 & 1 & 0 & 0 & 0 & 1 & 1 & 0 & 1 & 1 & 1 & 0 & 0 & 0 & 0 & 0 & 1 & 0 & 0 & 0 & 0 & 0 & 0 & 0 & 0 & 1 & 0 & 1 & 0 & 1 & 1 & 0 & 0 & 1 & 0 & 0 & 1 & 1 & 0 & 1 \\
0 & 1 & 1 & 0 & 0 & 0 & 0 & 1 & 1 & 1 & 0 & 1 & 1 & 0 & 0 & 0 & 1 & 1 & 0 & 1 & 1 & 1 & 0 & 0 & 0 & 1 & 0 & 1 & 0 & 0 & 0 & 0 & 0 & 0 & 0 & 0 & 1 & 0 & 1 & 0 & 1 & 1 & 0 & 0 & 1 & 0 & 0 & 1 & 1 & 0 \\
0 & 0 & 1 & 1 & 0 & 0 & 0 & 0 & 1 & 1 & 1 & 0 & 1 & 1 & 0 & 0 & 0 & 1 & 1 & 0 & 1 & 1 & 1 & 0 & 0 & 0 & 1 & 0 & 1 & 0 & 0 & 0 & 0 & 0 & 0 & 0 & 0 & 1 & 0 & 1 & 0 & 1 & 1 & 0 & 0 & 1 & 0 & 0 & 1 & 1 \\
0 & 0 & 0 & 1 & 1 & 0 & 0 & 0 & 0 & 1 & 1 & 1 & 0 & 1 & 1 & 0 & 0 & 0 & 1 & 1 & 0 & 1 & 1 & 1 & 0 & 1 & 0 & 1 & 0 & 1 & 0 & 0 & 0 & 0 & 0 & 0 & 0 & 0 & 1 & 0 & 1 & 0 & 1 & 1 & 0 & 0 & 1 & 0 & 0 & 1 \\
0 & 0 & 0 & 0 & 1 & 1 & 0 & 0 & 0 & 0 & 1 & 1 & 1 & 0 & 1 & 1 & 0 & 0 & 0 & 1 & 1 & 0 & 1 & 1 & 1 & 1 & 1 & 0 & 1 & 0 & 1 & 0 & 0 & 0 & 0 & 0 & 0 & 0 & 0 & 1 & 0 & 1 & 0 & 1 & 1 & 0 & 0 & 1 & 0 & 0 \\
1 & 0 & 0 & 0 & 0 & 1 & 1 & 0 & 0 & 0 & 0 & 1 & 1 & 1 & 0 & 1 & 1 & 0 & 0 & 0 & 1 & 1 & 0 & 1 & 1 & 0 & 1 & 1 & 0 & 1 & 0 & 1 & 0 & 0 & 0 & 0 & 0 & 0 & 0 & 0 & 1 & 0 & 1 & 0 & 1 & 1 & 0 & 0 & 1 & 0 \\
1 & 1 & 0 & 0 & 0 & 0 & 1 & 1 & 0 & 0 & 0 & 0 & 1 & 1 & 1 & 0 & 1 & 1 & 0 & 0 & 0 & 1 & 1 & 0 & 1 & 0 & 0 & 1 & 1 & 0 & 1 & 0 & 1 & 0 & 0 & 0 & 0 & 0 & 0 & 0 & 0 & 1 & 0 & 1 & 0 & 1 & 1 & 0 & 0 & 1 \\
1 & 1 & 1 & 0 & 0 & 0 & 0 & 1 & 1 & 0 & 0 & 0 & 0 & 1 & 1 & 1 & 0 & 1 & 1 & 0 & 0 & 0 & 1 & 1 & 0 & 1 & 0 & 0 & 1 & 1 & 0 & 1 & 0 & 1 & 0 & 0 & 0 & 0 & 0 & 0 & 0 & 0 & 1 & 0 & 1 & 0 & 1 & 1 & 0 & 0 \\
0 & 1 & 1 & 1 & 0 & 0 & 0 & 0 & 1 & 1 & 0 & 0 & 0 & 0 & 1 & 1 & 1 & 0 & 1 & 1 & 0 & 0 & 0 & 1 & 1 & 0 & 1 & 0 & 0 & 1 & 1 & 0 & 1 & 0 & 1 & 0 & 0 & 0 & 0 & 0 & 0 & 0 & 0 & 1 & 0 & 1 & 0 & 1 & 1 & 0 \\
1 & 0 & 1 & 1 & 1 & 0 & 0 & 0 & 0 & 1 & 1 & 0 & 0 & 0 & 0 & 1 & 1 & 1 & 0 & 1 & 1 & 0 & 0 & 0 & 1 & 0 & 0 & 1 & 0 & 0 & 1 & 1 & 0 & 1 & 0 & 1 & 0 & 0 & 0 & 0 & 0 & 0 & 0 & 0 & 1 & 0 & 1 & 0 & 1 & 1 \\
1 & 1 & 0 & 1 & 1 & 1 & 0 & 0 & 0 & 0 & 1 & 1 & 0 & 0 & 0 & 0 & 1 & 1 & 1 & 0 & 1 & 1 & 0 & 0 & 0 & 1 & 0 & 0 & 1 & 0 & 0 & 1 & 1 & 0 & 1 & 0 & 1 & 0 & 0 & 0 & 0 & 0 & 0 & 0 & 0 & 1 & 0 & 1 & 0 & 1 \\
0 & 1 & 1 & 0 & 1 & 1 & 1 & 0 & 0 & 0 & 0 & 1 & 1 & 0 & 0 & 0 & 0 & 1 & 1 & 1 & 0 & 1 & 1 & 0 & 0 & 1 & 1 & 0 & 0 & 1 & 0 & 0 & 1 & 1 & 0 & 1 & 0 & 1 & 0 & 0 & 0 & 0 & 0 & 0 & 0 & 0 & 1 & 0 & 1 & 0 \\ \hline
0 & 1 & 0 & 1 & 0 & 0 & 0 & 0 & 0 & 0 & 0 & 0 & 1 & 0 & 1 & 0 & 1 & 1 & 0 & 0 & 1 & 0 & 0 & 1 & 1 & 0 & 1 & 0 & 0 & 1 & 0 & 0 & 0 & 1 & 1 & 1 & 1 & 0 & 0 & 1 & 1 & 1 & 1 & 0 & 0 & 0 & 1 & 0 & 0 & 1 \\
1 & 0 & 1 & 0 & 1 & 0 & 0 & 0 & 0 & 0 & 0 & 0 & 0 & 1 & 0 & 1 & 0 & 1 & 1 & 0 & 0 & 1 & 0 & 0 & 1 & 1 & 0 & 1 & 0 & 0 & 1 & 0 & 0 & 0 & 1 & 1 & 1 & 1 & 0 & 0 & 1 & 1 & 1 & 1 & 0 & 0 & 0 & 1 & 0 & 0 \\
1 & 1 & 0 & 1 & 0 & 1 & 0 & 0 & 0 & 0 & 0 & 0 & 0 & 0 & 1 & 0 & 1 & 0 & 1 & 1 & 0 & 0 & 1 & 0 & 0 & 0 & 1 & 0 & 1 & 0 & 0 & 1 & 0 & 0 & 0 & 1 & 1 & 1 & 1 & 0 & 0 & 1 & 1 & 1 & 1 & 0 & 0 & 0 & 1 & 0 \\
0 & 1 & 1 & 0 & 1 & 0 & 1 & 0 & 0 & 0 & 0 & 0 & 0 & 0 & 0 & 1 & 0 & 1 & 0 & 1 & 1 & 0 & 0 & 1 & 0 & 0 & 0 & 1 & 0 & 1 & 0 & 0 & 1 & 0 & 0 & 0 & 1 & 1 & 1 & 1 & 0 & 0 & 1 & 1 & 1 & 1 & 0 & 0 & 0 & 1 \\
0 & 0 & 1 & 1 & 0 & 1 & 0 & 1 & 0 & 0 & 0 & 0 & 0 & 0 & 0 & 0 & 1 & 0 & 1 & 0 & 1 & 1 & 0 & 0 & 1 & 1 & 0 & 0 & 1 & 0 & 1 & 0 & 0 & 1 & 0 & 0 & 0 & 1 & 1 & 1 & 1 & 0 & 0 & 1 & 1 & 1 & 1 & 0 & 0 & 0 \\
1 & 0 & 0 & 1 & 1 & 0 & 1 & 0 & 1 & 0 & 0 & 0 & 0 & 0 & 0 & 0 & 0 & 1 & 0 & 1 & 0 & 1 & 1 & 0 & 0 & 0 & 1 & 0 & 0 & 1 & 0 & 1 & 0 & 0 & 1 & 0 & 0 & 0 & 1 & 1 & 1 & 1 & 0 & 0 & 1 & 1 & 1 & 1 & 0 & 0 \\
0 & 1 & 0 & 0 & 1 & 1 & 0 & 1 & 0 & 1 & 0 & 0 & 0 & 0 & 0 & 0 & 0 & 0 & 1 & 0 & 1 & 0 & 1 & 1 & 0 & 0 & 0 & 1 & 0 & 0 & 1 & 0 & 1 & 0 & 0 & 1 & 0 & 0 & 0 & 1 & 1 & 1 & 1 & 0 & 0 & 1 & 1 & 1 & 1 & 0 \\
0 & 0 & 1 & 0 & 0 & 1 & 1 & 0 & 1 & 0 & 1 & 0 & 0 & 0 & 0 & 0 & 0 & 0 & 0 & 1 & 0 & 1 & 0 & 1 & 1 & 0 & 0 & 0 & 1 & 0 & 0 & 1 & 0 & 1 & 0 & 0 & 1 & 0 & 0 & 0 & 1 & 1 & 1 & 1 & 0 & 0 & 1 & 1 & 1 & 1 \\
1 & 0 & 0 & 1 & 0 & 0 & 1 & 1 & 0 & 1 & 0 & 1 & 0 & 0 & 0 & 0 & 0 & 0 & 0 & 0 & 1 & 0 & 1 & 0 & 1 & 1 & 0 & 0 & 0 & 1 & 0 & 0 & 1 & 0 & 1 & 0 & 0 & 1 & 0 & 0 & 0 & 1 & 1 & 1 & 1 & 0 & 0 & 1 & 1 & 1 \\
1 & 1 & 0 & 0 & 1 & 0 & 0 & 1 & 1 & 0 & 1 & 0 & 1 & 0 & 0 & 0 & 0 & 0 & 0 & 0 & 0 & 1 & 0 & 1 & 0 & 1 & 1 & 0 & 0 & 0 & 1 & 0 & 0 & 1 & 0 & 1 & 0 & 0 & 1 & 0 & 0 & 0 & 1 & 1 & 1 & 1 & 0 & 0 & 1 & 1 \\
0 & 1 & 1 & 0 & 0 & 1 & 0 & 0 & 1 & 1 & 0 & 1 & 0 & 1 & 0 & 0 & 0 & 0 & 0 & 0 & 0 & 0 & 1 & 0 & 1 & 1 & 1 & 1 & 0 & 0 & 0 & 1 & 0 & 0 & 1 & 0 & 1 & 0 & 0 & 1 & 0 & 0 & 0 & 1 & 1 & 1 & 1 & 0 & 0 & 1 \\
1 & 0 & 1 & 1 & 0 & 0 & 1 & 0 & 0 & 1 & 1 & 0 & 1 & 0 & 1 & 0 & 0 & 0 & 0 & 0 & 0 & 0 & 0 & 1 & 0 & 1 & 1 & 1 & 1 & 0 & 0 & 0 & 1 & 0 & 0 & 1 & 0 & 1 & 0 & 0 & 1 & 0 & 0 & 0 & 1 & 1 & 1 & 1 & 0 & 0 \\
0 & 1 & 0 & 1 & 1 & 0 & 0 & 1 & 0 & 0 & 1 & 1 & 0 & 1 & 0 & 1 & 0 & 0 & 0 & 0 & 0 & 0 & 0 & 0 & 1 & 0 & 1 & 1 & 1 & 1 & 0 & 0 & 0 & 1 & 0 & 0 & 1 & 0 & 1 & 0 & 0 & 1 & 0 & 0 & 0 & 1 & 1 & 1 & 1 & 0 \\
1 & 0 & 1 & 0 & 1 & 1 & 0 & 0 & 1 & 0 & 0 & 1 & 1 & 0 & 1 & 0 & 1 & 0 & 0 & 0 & 0 & 0 & 0 & 0 & 0 & 0 & 0 & 1 & 1 & 1 & 1 & 0 & 0 & 0 & 1 & 0 & 0 & 1 & 0 & 1 & 0 & 0 & 1 & 0 & 0 & 0 & 1 & 1 & 1 & 1 \\
0 & 1 & 0 & 1 & 0 & 1 & 1 & 0 & 0 & 1 & 0 & 0 & 1 & 1 & 0 & 1 & 0 & 1 & 0 & 0 & 0 & 0 & 0 & 0 & 0 & 1 & 0 & 0 & 1 & 1 & 1 & 1 & 0 & 0 & 0 & 1 & 0 & 0 & 1 & 0 & 1 & 0 & 0 & 1 & 0 & 0 & 0 & 1 & 1 & 1 \\
0 & 0 & 1 & 0 & 1 & 0 & 1 & 1 & 0 & 0 & 1 & 0 & 0 & 1 & 1 & 0 & 1 & 0 & 1 & 0 & 0 & 0 & 0 & 0 & 0 & 1 & 1 & 0 & 0 & 1 & 1 & 1 & 1 & 0 & 0 & 0 & 1 & 0 & 0 & 1 & 0 & 1 & 0 & 0 & 1 & 0 & 0 & 0 & 1 & 1 \\
0 & 0 & 0 & 1 & 0 & 1 & 0 & 1 & 1 & 0 & 0 & 1 & 0 & 0 & 1 & 1 & 0 & 1 & 0 & 1 & 0 & 0 & 0 & 0 & 0 & 1 & 1 & 1 & 0 & 0 & 1 & 1 & 1 & 1 & 0 & 0 & 0 & 1 & 0 & 0 & 1 & 0 & 1 & 0 & 0 & 1 & 0 & 0 & 0 & 1 \\
0 & 0 & 0 & 0 & 1 & 0 & 1 & 0 & 1 & 1 & 0 & 0 & 1 & 0 & 0 & 1 & 1 & 0 & 1 & 0 & 1 & 0 & 0 & 0 & 0 & 1 & 1 & 1 & 1 & 0 & 0 & 1 & 1 & 1 & 1 & 0 & 0 & 0 & 1 & 0 & 0 & 1 & 0 & 1 & 0 & 0 & 1 & 0 & 0 & 0 \\
0 & 0 & 0 & 0 & 0 & 1 & 0 & 1 & 0 & 1 & 1 & 0 & 0 & 1 & 0 & 0 & 1 & 1 & 0 & 1 & 0 & 1 & 0 & 0 & 0 & 0 & 1 & 1 & 1 & 1 & 0 & 0 & 1 & 1 & 1 & 1 & 0 & 0 & 0 & 1 & 0 & 0 & 1 & 0 & 1 & 0 & 0 & 1 & 0 & 0 \\
0 & 0 & 0 & 0 & 0 & 0 & 1 & 0 & 1 & 0 & 1 & 1 & 0 & 0 & 1 & 0 & 0 & 1 & 1 & 0 & 1 & 0 & 1 & 0 & 0 & 0 & 0 & 1 & 1 & 1 & 1 & 0 & 0 & 1 & 1 & 1 & 1 & 0 & 0 & 0 & 1 & 0 & 0 & 1 & 0 & 1 & 0 & 0 & 1 & 0 \\
0 & 0 & 0 & 0 & 0 & 0 & 0 & 1 & 0 & 1 & 0 & 1 & 1 & 0 & 0 & 1 & 0 & 0 & 1 & 1 & 0 & 1 & 0 & 1 & 0 & 0 & 0 & 0 & 1 & 1 & 1 & 1 & 0 & 0 & 1 & 1 & 1 & 1 & 0 & 0 & 0 & 1 & 0 & 0 & 1 & 0 & 1 & 0 & 0 & 1 \\
0 & 0 & 0 & 0 & 0 & 0 & 0 & 0 & 1 & 0 & 1 & 0 & 1 & 1 & 0 & 0 & 1 & 0 & 0 & 1 & 1 & 0 & 1 & 0 & 1 & 1 & 0 & 0 & 0 & 1 & 1 & 1 & 1 & 0 & 0 & 1 & 1 & 1 & 1 & 0 & 0 & 0 & 1 & 0 & 0 & 1 & 0 & 1 & 0 & 0 \\
1 & 0 & 0 & 0 & 0 & 0 & 0 & 0 & 0 & 1 & 0 & 1 & 0 & 1 & 1 & 0 & 0 & 1 & 0 & 0 & 1 & 1 & 0 & 1 & 0 & 0 & 1 & 0 & 0 & 0 & 1 & 1 & 1 & 1 & 0 & 0 & 1 & 1 & 1 & 1 & 0 & 0 & 0 & 1 & 0 & 0 & 1 & 0 & 1 & 0 \\
0 & 1 & 0 & 0 & 0 & 0 & 0 & 0 & 0 & 0 & 1 & 0 & 1 & 0 & 1 & 1 & 0 & 0 & 1 & 0 & 0 & 1 & 1 & 0 & 1 & 0 & 0 & 1 & 0 & 0 & 0 & 1 & 1 & 1 & 1 & 0 & 0 & 1 & 1 & 1 & 1 & 0 & 0 & 0 & 1 & 0 & 0 & 1 & 0 & 1 \\
1 & 0 & 1 & 0 & 0 & 0 & 0 & 0 & 0 & 0 & 0 & 1 & 0 & 1 & 0 & 1 & 1 & 0 & 0 & 1 & 0 & 0 & 1 & 1 & 0 & 1 & 0 & 0 & 1 & 0 & 0 & 0 & 1 & 1 & 1 & 1 & 0 & 0 & 1 & 1 & 1 & 1 & 0 & 0 & 0 & 1 & 0 & 0 & 1 & 0
\end{array}
\right)
$}
\end{equation}

The full automorphism group of the found graph was computed using the GAP system~\cite{GAP4} and is isomorphic to the group~$C_{25}\rtimes C_{4}$; the order of the group is 100.

Strongly regular graphs with the parameter set $(50, 21, 8, 9)$ have already been found. However, the graphs with such a parameter set mentioned on the website~\cite{spence_srg} and in the papers~\cite{maksimovic_enumeration_2018} and \cite{maksimovic_regular_2023} have different automorphism groups. Therefore, there is reason to believe that the found graph is new (in the sense of being mentioned in the scientific literature).

By means of computer calculation, it was established that the clique number of this graph $G$ is $\omega(G) = 4$, the independence number is $\alpha(G) = 7$, and the chromatic number is \mbox{$\chi(G) = 8$}.

The found strongly regular graph generates one of the known regular two-graphs on 50 vertices; the order of the automorphism group of this two-graph is $117600$ (see~\cite{maksimovic_regular_2023}).

The obtained graph can be realized from the Paley graph $P(49)$. To do this, we consider the finite field $\mathbb{F}_{49} \cong \mathbb{Z}_7[x]/(x^2+6x+3)$. As a primitive element of the field, we choose $\alpha = x$, where $x$ is a root of the defining irreducible polynomial. Next, we construct the graph $P(49)$ based on this field, add an isolated vertex, and perform the Seidel switching operation in the resulting graph on 50 vertices with the switching set
\begin{multline}
W = \{ 1, 3, 5, x + 4, x + 1, 2x + 4, 6x + 1, 6x + 4, 6x + 5, 2x + 1, 3x + 1,\\ 4x + 1, 5x + 1, x + 6, 4x, 4x + 2, 3x + 4, 5x, 3x + 2, 5x + 5, 2x + 5 \}.
\end{multline}

\section{Conclusion}
\label{sec:conclusion}
In this paper, two results are obtained: it is shown that there exists an infinite sequence of Paley graphs that (excluding the zero vertex) can be partitioned into subgraphs isomorphic to the Petersen graph; a new strongly regular graph with parameters $(50, 21, 8, 9)$ is constructed. In both cases, the idea of pairing cosets of the multiplicative group of a finite field, considered as subsets of the vertices of a Paley graph, was used.

The author considers it promising to search for other strongly regular graphs by combining two or more orbits into groups within a Paley graph.

\printbibliography

\end{document}